\documentclass[11pt]{amsart}

\usepackage{geometry}                
\usepackage{amsmath, amssymb, amsthm,amsfonts}
\usepackage{xcolor}
\usepackage[applemac]{inputenc}
\definecolor{myurlcolor}{rgb}{0,0,0.4}
\definecolor{mycitecolor}{rgb}{0,0.5,0}
\definecolor{myrefcolor}{rgb}{0.5,0,0}
\usepackage[pagebackref,colorlinks,linkcolor=myrefcolor,citecolor=mycitecolor,urlcolor=myurlcolor,hypertexnames=true]{hyperref}
\usepackage{amsrefs}
\usepackage{graphicx}
\usepackage{pgfplots}
\usepackage{enumitem}
\usepackage{microtype}
\usepackage{cleveref}
\usepgfplotslibrary{fillbetween}

\theoremstyle{plain}
\newtheorem{theorem}{Theorem}[section]

\newtheorem{corollary}[theorem]{Corollary}

\theoremstyle{remark}

\newtheorem{remark}[theorem]{Remark}

\theoremstyle{definition}
\newtheorem{definition}[theorem]{Definition}
\newtheorem{example}[theorem]{Example}

\title{A note on polynomial equations over algebras}
\author{Maximilian Illmer and Tim Netzer}
\date{\today}                                           

\begin{document}
\begin{abstract}We provide sufficient conditions for systems of polynomial equations over general (real or complex) algebras to have a solution. This generalizes known results on quaternions, octonions and matrix algebras. We also generalize the fundamental theorem of algebra for quaternions to polynomials with two monomials in the leading form, while showing that  it fails for three.
\end{abstract}
\maketitle

%%%%%%%%%%%%%%%%%%%%%%%%%%%%%%%%%%%%%%%%%%%%%%%%%%%%%%%%%%%%%%%%%%

\section{Introduction} By the fundamental theorem of algebra,  every non-constant polynomial has a zero over $\mathbb C$. Over the reals, only polynomials of odd degree are guaranteed to have  zeros, by the intermediate value theorem. But how about other real or complex algebras? 

A  beautiful result for quaternions and octonions was proven in \cite{eil}, see also \cite{roo}: Every  non-constant polynomial, whose leading form consists of only one monomial, has a zero. So in this regard, quaternions/octonions behave  like complex numbers, although they are  non-commutative/even non-associative. Another such result is \cite{neth}: Every odd-degree Hermitian polynomial over the complex matrix algebra, whose leading form is non-degenerate, has a zero in Hermitian matrices. This resembles exactly the behavior of real numbers.

In this note, we  study and  extend these results to general algebras over algebraically closed and real closed fields, and to systems of multivariate polynomial equations.

First, for algebras over an algebraically closed field, \Cref{thm_complex} gives a sufficient condition for a system of non-constant polynomial equations to have a solution. The two crucial requirements are non-degenerateness of the leading form, and some finite-dimensionality condition on the polynomials. Both conditions are necessary for the result to hold. In the proof we reduce the problem to a suitable system of classical polynomial equations, whose solvability is well-known.

Second, we prove the same result for algebras over real closed fields, given that the polynomials are of odd degree (\Cref{thm_real}). Here we give two different proofs. One reduces the problem again to classical polynomials over the ground field, the other uses methods from algebraic topology as in  \cite{eil, neth} and one of the standard proofs of the fundamental theorem of algebra. This proof is a simplification of the one in \cite{neth}, and proves a much more general result. In particular, it also applies to quaternions and octonions, and yields a significant generalization of \cite{eil} in case of odd degree.

Finally, we consider quaternions and octonions  in more detail.  For even degree polynomials, we  extend the results from \cite{eil,roo} to polynomials with two monomials in the leading form. Maybe surprisingly, we then show that the result fails for leading forms of three monomials. So there are indeed  (non-degenerate) polynomials over quaternions/octonions without a zero.

This paper is structured as follows. In \Cref{sec:prel} we explain the setup of polynomials over algebras. In \Cref{sec:main} we prove the general results   \Cref{thm_complex} and \Cref{thm_real}. \Cref{sec:quat}  contains the results about quaternions and octonions.

\section{Polynomials over algebras}\label{sec:prel}

Let $\mathcal A$ be an algebra over the infinite field $k$. By definition, this means that $\mathcal A$ is a $k$-vector space, equipped with a $k$-bilinear map $\mathcal A\times\mathcal A\to\mathcal A$, called the {\it multiplication} of $\mathcal A$. We do {\it not} assume the multiplication to be associative or commutative, nor do we assume the existence of a unit $1$.  For the product of two elements $a,b\in\mathcal A$ we will just write $ab$. In this setup we now define (multivariate) polynomial maps.
\begin{definition}
 ($i$) {\it Monomial maps} are defined recursively:
\begin{itemize}
\item All constant maps $\mathcal A\to\mathcal A$ are monomial.
\item  For all $n\geqslant 1$ and $i=1,\ldots, n$, the projection $\pi_i\colon\mathcal A^n\to \mathcal A$ to the $i$-th component is monomial (note that this includes the identity ${\rm id_{\mathcal A}}\colon\mathcal A\to\mathcal A)$.
\item The multiplication $\mathcal A\times\mathcal A\to\mathcal A$ is monomial.
\item If $p_1,\ldots, p_m\colon \mathcal A^n\to\mathcal A$ and $q\colon\mathcal A^m\to \mathcal A$ are monomial, then so is 
\begin{align*}
q(p_1,\ldots, p_m)\colon \mathcal A^n&\to\mathcal A \\ a&\mapsto q(p_1(a),\ldots, p_m(a)).
\end{align*}
\end{itemize}
 
($ii$)  A {\it polynomial map}  is a finite $k$-linear combination of monomial maps (inside the $k$-vector space of maps from $\mathcal A^n$ to $\mathcal A$).

($iii$) A {\it zero} of a map $p\colon\mathcal A^n\to \mathcal A$ is an element $a\in\mathcal A^n$ with $p(a)=0_{\mathcal A}$. 
\end{definition}

\begin{remark}
($i$) Monomial maps can also be defined as follows. A {\it non-associative word} is a finite string of variables and elements from $\mathcal A$, but equipped with a sensible bracketing. Sensible in this context means  that it possible to actually compute the expression, whenever the variables are replaced by algebra elements. One example is $$ (xy)(((ax)(yb))x)$$ where $a,b\in \mathcal A$ and $x,y$ are variables. Plugging in elements for the variables then defines a map, and the so-defined maps are precisely the  monomial maps in our above sense. Since the notion of a non-associative word is pretty technical to define exactly, and also different words can define the same map, we will stick to the above definition  of  monomial and polynomial maps.
 
($ii$) If $\mathcal A$ is associative, monomial maps become easier to express. One can  use a word composed of variables and algebra elements without brackets. If the algebra is even unital, each monomial map has a representation as $$(a_1,\ldots, a_n)\mapsto c_0a_{i_1}c_1a_{i_2}c_2\cdots c_{d-1}a_{i_d}c_{d}$$ with coefficients $c_j\in\mathcal A$. If the algebra is not unital, one also has to  allow   $a_i$'s without separation by  coefficients $c_i$.

($iii$) If $\mathcal A$ is associative and commutative, monomial maps take the classical form $$(a_1,\ldots, a_n)\mapsto c a_1^{d_1}\cdots a_n^{d_n}$$ with $c\in\mathcal A$. Again, if $\mathcal A$ is not unital, one also has to allow the same rule without a coefficient $c$ in front.
\end{remark}

\begin{definition}
Using that $k$ is infinite, its easy to see that  nonzero monomial maps  are {\it homogeneous} of a unique {\it degree} $d$, i.e.\ fulfill $$p(\lambda a)=\lambda^dp(a)$$ for all $\lambda\in k, a\in\mathcal A^n.$ In fact they are even {\it multi-homogeneous}, i.e.\ homogeneous with respect to every single argument separately. The zero map is defined to be homogeneous of every degree. Every polynomial map is a unique finite sum of nonzero homogeneous polynomial maps of different degrees, called its {\it homogeneous parts} or {\it forms}. The degree of a polynomial map is the largest degree among its homogeneous parts, and its {\it leading form} is the  homogeneous part of largest degree. 
\end{definition}

\begin{example}
A polynomial map in one variable over an associative algebra $\mathcal A$  is for example given by $$a\mapsto c_1ac_2ac_3 +c_4ac_5ac_6 + c_7a + ac_8 +c_9$$ where all $c_i\in \mathcal A.$  The first two terms constitute the leading form, i.e.\ the homogeneous part of degree $2$, the third and fourth term are the homogeneous part of degree $1$, and $c_9$ is the homogeneous part of degree $0$.   In two variables, one example of a polynomial map is $$(a,b)\mapsto c_1ac_2bc_3ac_4 + bc_5b + ac_6b + c_7.$$
\end{example}

\begin{remark}
Let   $p\colon \mathcal A^n\to\mathcal A$ be  a polynomial map. It is easily seen along the recursive definition, that  for any finite-dimensional subspace  $\mathcal H\subseteq \mathcal A,$  there exists another finite-dimensional subspace $\mathcal H'\subseteq\mathcal A$ with $p(\mathcal H^n)\subseteq \mathcal H'.$ 
Also, if   $b_1,\ldots, b_d$ is a basis of $\mathcal H$ and $b_1',\ldots, b_e'$ is a basis of $\mathcal H'$, then 
$$p(a)=\sum_{j=1}^eh_j(a)b_j'$$ for all $a\in\mathcal H^n$, where the $h_j$ are classical polynomials over $k$ in the variables $\lambda_{k\ell}$, when $a$ is expressed in the basis as follows: $$\mathcal H^n\ni a=(a_1,\ldots, a_n)=\left(\sum_{k=1}^d \lambda_{k1}b_k, \ldots, \sum_{k=1}^d \lambda_{kn}b_k\right).$$ The homogeneous parts of $p$ correspond to the homogeneous parts of the $h_j$ of the same degree. 
\end{remark}

%%%%%%%%%%%%%%%%%%%%%%%%%%%%%%%%%%%%%%%%%%%%%%%%%%%%%%%%%%%%%%%%%%

\section{Solutions to polynomial equations}\label{sec:main}

The following is our main result on solutions of polynomial equations over algebraically closed fields.
\begin{theorem}\label{thm_complex} Let  $\mathcal A$ be an algebra over the algebraically closed field $k$, and let $$p_1,\ldots, p_n\colon\mathcal A^n\to\mathcal A$$ be  polynomial maps of positive degree.
Assume $\mathcal H\subseteq\mathcal A$ is a finite-dimensional subspace, such that  the leading forms of the $p_i$ are non-degenerate on $\mathcal H$ (i.e.\ they have no \emph{common} zero in $\mathcal H^n\setminus\{0\}$). Further assume there exists another subspace $\mathcal H'\subseteq\mathcal A$ with $\dim(\mathcal H')\leqslant \dim(\mathcal H)$ and $$p_i(\mathcal H^n)\subseteq \mathcal H' \quad\mbox{for } i=1,\ldots, n.$$ Then $p_1,\ldots, p_n$ have a common zero in $\mathcal H^n$.\end{theorem}
\begin{proof}
Choose bases $b_1,\ldots, b_d$ of $\mathcal H$ and $b_1',\ldots, b_e'$ of $\mathcal H'$,  and obtain  $$p_i(a)=\sum_{j=1}^e h_{ji}(a)b_j',$$ where all $h_{ji}$ are classical polynomials over $k$, when expressed in the coefficients $\lambda_{k\ell}$ of $a\in\mathcal H^n$ with respect to $b_1,\ldots,b_d$. 

Now setting all $p_i$ to zero on $\mathcal H^n$ means setting all $h_{ji}$ to zero, i.e.\ we solve a system of $ne\leqslant nd$ many polynomial equations in $nd$ variables. The homogeneous parts of degree $\deg(p_i)$ of each  $h_{ji}$ must be nonzero, because otherwise we had strictly less than $nd$ of these nontrivial  homogeneous equations, which would  admit a nontrivial solution over $k$, see for example \cite{shaf}, Chapter 1, Section 6. This would give rise to a common zero of the leading forms of the $p_i$ in $\mathcal H^n\setminus\{0\}$, which we have excluded.  So all of our $ne$ many  polynomial equations have positive degree (we also see here that non-degeneratenes in fact implies $\dim(\mathcal H)=d=e=\dim(\mathcal H')$).

After homogenizing the whole system with an additional variable, we have $nd$ many (nonconstant) equations in $nd+1$ variables, and this system has a nontrivial solution over $k$, as above. But the value of the additional variable in such a solution must be nonzero, because otherwise we again had  a common zero   in $\mathcal H^n\setminus\{0\}$ of the leading forms. Thus we can assume that the additional variable takes the value $1$, and this gives rise to a common zero of the $p_i$ in $\mathcal H^n$, as desired. 
\end{proof}

An obvious corollary of \Cref{thm_complex} is the following:
\begin{corollary}
Let  $\mathcal A$ be a finite-dimensional algebra over the algebraically closed field $k$, and let $p_1,\ldots, p_n\colon\mathcal A^n\to\mathcal A$ be polynomial maps of positive degree, whose leading forms are non-degenerate on  $\mathcal A^n$. Then  $p_1,\ldots, p_n$ have a common zero in $\mathcal A^n$.
\end{corollary}

Our  next main result is a version for real closed fields.
\begin{theorem}\label{thm_real} Let  $\mathcal A$ be an algebra over the real closed field $k$, and let $$p_1,\ldots, p_n\colon\mathcal A^n\to\mathcal A$$ be  polynomial maps of odd degree.
Assume $\mathcal H\subseteq\mathcal A$ is a finite-dimensional subspace, such that  the leading forms of the $p_i$ are non-degenerate on $\mathcal H$ (i.e.\ they have no common zero in $\mathcal H^n\setminus\{0\}$). Further assume there exists another subspace $\mathcal H'\subseteq\mathcal A$ with $\dim(\mathcal H')\leqslant \dim(\mathcal H)$ and $$p_i(\mathcal H^n)\subseteq \mathcal H' \quad\mbox{for } i=1,\ldots, n.$$ Then $p_1,\ldots, p_n$ have a common zero in $\mathcal H^n$.
\end{theorem}
\begin{proof}[1st proof of \Cref{thm_real}]
One proceeds exactly as in the proof of \Cref{thm_complex}. This time one uses Theorem 4.3 in \cite{shaf} to obtain a solution of the derived polynomial equations over $k$. Note that the result is stated only for the field of real numbers there, but it immediately transfers to any real closed field by Tarski's Transfer Principle (see for example \cite{pd}). 
\end{proof} 

The above used argument on real solutions of  polynomial equations of odd degree relies on B\'{e}zout's Theorem. But it can also be proven with methods from algebraic topology. Since this proof applies directly in our more general setup, we include it here as an alternative. It is also the approach used in \cite{eil, neth} and a standard proof of the fundamental theorem of algebra. Note that it is again enough to prove the result for $k=\mathbb R$, since it transfers to any real closed field by Tarski's Transfer Principle.

\begin{proof}[2nd proof of \Cref{thm_real}]
Assume for contradiction that the polynomial maps $p_i$ do not have a common zero in $\mathcal H^n$. Then the continuous map 
\begin{align*}
p\colon \mathcal H^n&\to \mathcal H'^n \\ a &\mapsto (p_1(a),\ldots, p_n(a)) 
\end{align*}
does not attain the value zero. Thus for $t\in[0,\infty)$ we can consider the well-defined map 
\begin{align*}
p^{(t)}\colon \mathbb{S}&\to  \mathbb S' \\ a &\mapsto \frac{p(ta)}{\Vert p(ta)\Vert} 
\end{align*}
where $\Vert\cdot\Vert$ is any norm on the real vector space $\mathcal H'^n,$ and $\mathbb S,\mathbb S'$ are the unit spheres (choose an arbitrary norm in $\mathcal H^n$ as well). Denote the leading form of $p_i$ by $p_i^{\max}$, consider the map $p^{\max}=(p_1^{\max},\ldots, p_n^{\max})$ and 
\begin{align*}
p^{\infty}\colon \mathbb{S}&\to  \mathbb S' \\ a &\mapsto \frac{p^{\max}(a)}{\Vert p^{\max} (a)\Vert}, 
\end{align*}
which is also well-defined, since the leading forms are non-degenerate. The family $p^{(t)}$ clearly defines a homotopy between $p^{(0)}$ and $p^{\infty}$. Since $p^{(0)}$ is constant, it has topological degree $0$. On the other hand, $p^{\infty}$ is an odd map, which by Borsuk's Antipodal Theorem  has odd degree, see for example \cite{tom}, Theorem 10.6.3. This is a contradiction.
\end{proof}

\begin{remark} ($i$) The non-degenerateness of the leading forms is necessary. Consider the algebra $\mathcal A={\rm Mat}_2(k)$ of $2\times 2$-matrices, and the equation $$\left(\begin{array}{cc}1 & 0 \\0 & 0\end{array}\right)X+\left(\begin{array}{cc}1 & 0 \\0 & 1\end{array}\right)=0.$$ The leading form is degenerate, and the polynomial indeed has no zero in $\mathcal A$.

($ii$)
For infinite-dimensional subspaces $\mathcal H$, the above results  fail. For example, in the commutative algebra $\mathcal A=C([-1,1],\mathbb K)$ of continuous functions (where $\mathbb K=\mathbb R$ or $\mathbb C$), the equation $$fx-1=0$$ has no solution, if $f$ has a zero. But if $f$ has only isolated zeros, then the leading form is non-degenerate. 
\end{remark}

For $\mathcal A={\rm Mat}_m(\mathbb C)$ and its real subspace $\mathcal H={\rm Her}_m(\mathbb C)$ of Hermitian matrices, we obtain the main result of \cite{neth}. Note that the condition {\it self-adjoint} means $p(a)^*=p(a^*)$ for all $a\in\mathcal A$, which implies that $p$ maps $\mathcal H$ to itself.

\begin{corollary}[\cite{neth}]
Let $p\colon {\rm Mat}_m(\mathbb C)\to {\rm Mat}_m(\mathbb C)$ be a self-adjoint polynomial map of odd degree, whose leading form is non-degenerate on  ${\rm Her}_m(\mathbb C)$. Then $p$ has a zero  in ${\rm Her}_m(\mathbb C)$.
\end{corollary}

We also obtain a multivariate generalization of the odd-degree case of the fundamental theorem of algebra for quaternions/octonions \cite{eil,roo}. We will elaborate on this in more detail in the following section.

\begin{corollary}\label{cor:oct} Let $\mathcal A\in\{\mathbb H,\mathbb O\}$ be the real quaternion or octonion algebra. Let   \linebreak  $p_1,\ldots, p_n\colon \mathcal A^n\to\mathcal A$ be polynomial maps of odd degree, whose leading forms are non-degenerate on  $\mathcal A$ (i.e.\  have no common zero in $\mathcal A^n\setminus\{0\}$). Then $p_1,\ldots, p_n$ have a common zero in $\mathcal A^n$.

In particular, if the leading form of a single odd-degree polynomial map $p\colon \mathcal A\to\mathcal A$  consists of one monomial, then $p$ has a zero in $\mathcal A$.
\end{corollary}
\begin{proof}
The first statement is clear from \Cref{thm_real}. For the second, note that for $c_i\neq 0$ we have $$c_1ac_2a\cdots ac_d=0$$ only if $a=0$, since $\mathcal A$ is a division algebra. So a single nonzero monomial map is automatically non-degenerate.
\end{proof}

%%%%%%%%%%%%%%%%%%%%%%%%%%%%%%%%%%%%%%%%%%%%%%%%%%%%%%
\section{Quaternions and Octonions}\label{sec:quat}
For the real division algebras $\mathbb H, \mathbb O$ of quaternions and octonions, \Cref{cor:oct} provides  solutions to systems of polynomial equations of odd degree. But the fundamental theorem of algebra from \cite{eil} (see also \cite{roo}) provides solutions also in the even-degree case, at least for  univariate  polynomials with a single monomial leading form. We generalize this now to two monomials. We make  use of  the canonical norms on $\mathbb H, \mathbb O$ throughout.

\begin{theorem}\label{thm:quat}
Let $\mathcal A\in\{\mathbb H,\mathbb O\}$  and let $p\colon \mathcal A\to \mathcal A$ be a polynomial map of positive even degree, whose leading form is a sum of at most two monomials, and which is nondegenerate (i.e\ has no zero  in $\mathcal A\setminus\{0\}$). Then $p$ has a zero in $\mathcal A$.
\end{theorem}
\begin{proof}
We proceed as in the topological proof of \Cref{thm_real}, and only need to show that $p^\infty\colon\mathbb S \to\mathbb S$ has non-zero degree. So let $$p^{\max}\colon a\mapsto m_1(a) + m_2(a)$$ be the leading form of $p$, where $m_1,m_2$ are monomial maps. Let $\Vert m_i\Vert$ denote the product of the norms of all coefficients in $m_i$, so that $$\Vert m_i(a)\Vert=\Vert m_i\Vert \cdot \Vert a\Vert^{\deg(p)}$$ holds for all $a\in\mathcal A$. 
We first assume $\Vert m_1\Vert \geqslant \Vert m_2\Vert>0$ and show that for $t\in[0,1]$ the map $$h_t\colon a\mapsto m_1(a) + (1-t)m_2(a)$$ has no zero in $\mathcal A\setminus\{0\}$. For $t=0$ this follows from our assumption.  For $t>0$ and $h_t(a)=0$ we get  $$\Vert m_1\Vert \Vert a\Vert ^{\deg(p)}=\Vert m_1(a)\Vert=\Vert (t-1)m_2(a)\Vert= (1-t)\Vert m_2\Vert \Vert a\Vert^{\deg(p)}.$$ For $a\neq 0$ this contradicts  $\Vert m_1\Vert \geqslant \Vert m_2\Vert>0$ and thus proves the claim.

So  $h_t$ can be restricted to $\mathbb S$ and normalized, and thus provides a homotopy between $p^{\infty}$ and the map $$a\mapsto \frac{m_1(a)}{\Vert m_1(a)\Vert}.$$ We have thus reduced our problem to the case of a single monomial leading form. It was shown in \cite{eil} that such maps are homotopic to $$a\mapsto a^{\deg(p)}$$ which has degree $\deg(p)\neq 0$ (the proof applies to $\mathcal A=\mathbb O$ as well, see also \cite{roo}).
\end{proof}

Maybe surprisingly, the last result does {\it not} extend to more than two monomials in the leading form.

\begin{example}On $\mathcal A=\mathbb H$ consider the polynomial map $$p\colon a\mapsto c_0a^2+ac_1a+c_2ac_3a + c_4$$ where
$c_0=-1-{\rm i}+{\rm k}, c_1=-1-{\rm i}+{\rm j}-{\rm k}, c_2=-{\rm i}-{\rm j}+{\rm k}, c_3=-1+{\rm i}+{\rm j}+{\rm k}, c_4=6{\rm i}.$
To see that the leading form of $p$ is non-degenerate, we can cancel one copy of $a$ from the right, and consider the linear map $$a\mapsto c_0a+ac_1+c_2ac_3.$$ Its coefficient matrix  in real coordinates $a=a_1+a_2{\rm i} +a_3{\rm j} +a_4{\rm k}$ is  $$M=\left(\begin{array}{cccc}-1 & -1 & 0 & 1 \\-3 & -1 & 1 & 0 \\4 & 3 & -1 & -1 \\-1 & 0 & 1 & -5\end{array}\right),$$ which is nonsingular.

In real coordinates, the full map $p$ has the form
\begin{align*}p\colon \left(\begin{array}{c}a_1 \\a_2 \\a_3 \\a_4\end{array}\right) &\mapsto \left(\begin{array}{c}-a_1^2 + a_2^2 + a_3^2 + 5 a_4^2  + 2 a_1 a_2  - 4 a_1 a_3 - 4 a_2 a_3  + 2 a_1a_4  \\  - 3 a_1^2 - a_2^2  - a_3^2- a_4^2- 2 a_1 a_2  + 2 a_1 a_3 + 4 a_1 a_4 + 4 a_2 a_4 +  4 a_3 a_4 +6\\ 4 a_1^2 + 2 a_1 a_2 - 2 a_1 a_3 + 2 a_1 a_4 - 4 a_2 a_4 \\ a_1^2 + a_3^2+ a_4^2- 4 a_1 a_2 - 3 a_2^2 - 2 a_1 a_3  - 6 a_1 a_4 \end{array}\right). \end{align*} The ideal generated by these four coordinate functions contains the univariate polynomial $$216 + 324 a_4^2 - 927 a_4^4 - 148 a_4^6 + 2578 a_4^8,$$ which is in fact an element  of the reduced Gr\"obner basis with respect to the lexicographic ordering. If this polynomial had a real zero, then the polynomial $$q:=216 + 324x - 927x^2 - 148x^3 + 2578x^4$$ would have a positive real zero. But the solution formula for quartic polynomials shows that $q$ has no real zero at all. Another way to show this is to realize that $q$ is in fact a sum of squares:
$$q=q_1^2+q_2^2+q_3^2$$ where
\begin{align*}
q_1&=-\frac{49212319}{478950 \sqrt{6}}x^2+9 \sqrt{\frac{3}{2}}x+6 \sqrt{6}\\
q_2&= \frac{124008757}{10 \sqrt{188107618886}}x^2+\frac{1}{5} \sqrt{\frac{29456251}{6386}}x\\
q_3&= \frac{\sqrt{\frac{17668412385769586257}{88368753}}}{478950}x^2.
\end{align*}
Now $q$ vanishes at a real point if and only if all three $q_i$ vanish, which is clearly not true for any real point.

So  in total  we have shown that the polynomial $p$ has no zero in $\mathbb H$, although its leading form is non-degenerate.
\end{example}

\begin{remark}
The proofs of \cite{eil} and \Cref{thm:quat}  provide explicit  homotopies of the leading form of $p$ to the form $x^{\deg(p)}$, whose degree was shown to be $\deg(p)$ in \cite{eil}. It is crucial to have only non-degenerate forms along the homotopy. The  multivariate resultant of the four coordinate functions decides non-degenerateness (at least over $\mathbb C$). One might expect that its zero locus divides the space of coefficients of forms into several connected components, and thus homotopies to $x^{\deg(p)}$ cannot be expected in general. This is in fact true, as we conclude from the above example. But since homotopies {\it are} possible in certain simple cases,  the resultant will sometimes show a different behavior.

We illustrate this in a simple example. Consider linear forms of type $$\ell=c_0x+xc_1$$ with $c_0,c_1\in\mathbb H$. When expressing the coefficients as $c_i=c_{i1}+c_{i2}{\rm i}+c_{i3}{\rm j}+c_{i4}{\rm k}$ and the variable as $x=x_1+x_2{\rm i}+x_3{\rm j}+x_4{\rm k}$, the  four coordinate functions of $\ell$ are linear in the $x_i$, and the resultant is thus the determinant of the coefficient matrix 
$$M=\left(
\begin{array}{cccc}
 c_{01}+c_{11} & -(c_{02}+c_{12}) & -(c_{03}+c_{13}) & -(c_{04}+c_{14})\\
 c_{02}+c_{12} & c_{01}+c_{11} & c_{14}-c_{04}& c_{03}-c_{13} \\
 c_{03}+c_{13} & c_{04}-c_{14}& c_{01}+c_{11}& c_{12}-c_{02} \\
 c_{04}+c_{14}& c_{13}-c_{03} & c_{02}-c_{12}& c_{01}+c_{11}
\end{array}
\right).
$$
A direct computation shows that \begin{align*}\det(M)=&(c_{01} + c_{11})^2\left((c_{01} + c_{11})^2 + 2(c_{02}^2 + c_{03}^2 + c_{04}^2  + 
   c_{12}^2 + c_{13}^2 +  c_{14}^2)\right)\\ &+ (c_{02}^2 + c_{03}^2 + c_{04}^2 - c_{12}^2 - c_{13}^2 - c_{14}^2)^2 
   \end{align*} is a sum of squares. Thus the form $\ell$ is degenerate  if and only if $$c_{01} =- c_{11}\ \mbox{ and }\  c_{02}^2 + c_{03}^2 + c_{04}^2 =c_{12}^2 + c_{13}^2 + c_{14}^2.$$ This defines an algebraic subset of the coefficient space $\mathbb R^8$ of codimension $2$, and does not divide it into several connected components. A similar behavior must occur for forms of higher degree, as long as there are at most two monomials. 
   
   For non-degenerate forms of odd degree,  the proof of \Cref{thm_real} shows that the degree of the induced mapping is non-zero, independent of the number of monomials. It does however {\it not} show that the forms are homotopic to $x^{\deg(p)},$ and thus does not imply  that the resultant must show such a behavior.
\end{remark}

%Now for $$h:=\sqrt{2578}y^2  +37\sqrt{{2}{1289}}y - \frac{121}{10}$$ one computes $$q-h^2=\left(20 \sqrt{2578} -\frac{1197641}{1289}\right)y^2 + \left(324 - 740 \sqrt{\frac{2}{1289}}\right)y +116.$$ This latter quadratic polynomial has a negative discriminant

%If this polynomial had a real zero, then the polynomial $$q:=216 + 324x - 927x^2 - 148x^3 + 2578x^4$$ would have a positive real zero. Since $q(x+\frac13)$ has only positive coefficients, such a zero would lie in in $(0,1/3)$.
%Since $q(0)>0$ and $q(1/3)>0$, the derivative $$q'= 324 - 1854x - 444x^2 + 10312x^3$$ would have a zero in this interval as well. From $q'(0)>0$ and $q'(1/3)>0$ we then obtain that 

%%%%%%%%%%%%%%%%%%%%%%%%%%%%%%%%%%%%%%%%%%%%%%%%%%%%%%%%%%%%%%%%%%
%group algebras, ideals,  Lie algebras, Quaternionen/Coquaternionen/Splitquaternionen/Oktonionen bei geradem Grad $xp(x)=0$ und $yxy-1=0$, was wenn $x$ und $x^*$ oder $\overline x$ auftreten d\"urfen ($x^*=axb$, also schon bewiesen!?)
%%%%%%%%%%%%%%%%%%%%%%%%%%%%%%%%%%%%%%%%%%%%%%%%%%%%%%%%%%%%%%%%%%
\bibliography{bib.bib}
\end{document}